\documentclass[11pt, oneside]{amsart}
\usepackage{latexsym}
\usepackage{enumitem}
\usepackage{layout}
\usepackage{amsfonts}
\usepackage{amssymb}
\usepackage{epsfig}
\usepackage{color}
\usepackage{cite}

\usepackage[normalem]{ulem}
\usepackage[perpage,para,symbol*]{footmisc}
\usepackage{fancyhdr}
\usepackage{times}
\usepackage{lipsum} 
\newcommand{\R}{\mathbb{R}}

\newcommand{\F}{\mathcal{F}}
\newcommand{\E}{\mathcal{E}}

\numberwithin{equation}{section}
\theoremstyle{plain}
\newtheorem{lem}[equation]{Lemma}
\newtheorem{thm}[equation]{Theorem}

\newtheorem{prop}[equation]{Proposition}
\newtheorem{cor}[equation]{Corollary}

\theoremstyle{definition}

\theoremstyle{remark}
\newtheorem{remark}[equation]{Remark}

\usepackage{verbatim}
\usepackage[margin=1.15in]{geometry}
\usepackage[hidelinks]{hyperref} 
\hypersetup{
    colorlinks=true,
}
\begin{document}
\title{Zarankiewicz's problem for semi-algebraic hypergraphs}
\author{Thao Do}
\address{Department of Mathematics, Massachusetts Institute of Technology, Cambridge MA 02139}
\email{thaodo@mit.edu}
\maketitle
\begin{abstract}
Zarankiewicz's problem asks for the largest possible number of edges in a graph that does not contain a $K_{u,u}$ subgraph for a fixed positive integer $u$. Recently, Fox, Pach, Sheffer, Sulk and Zahl \cite{Fox} considered this problem for semi-algebraic graphs, where vertices are points in $\mathbb{R}^d$ and edges are defined by some semi-algebraic relations. In this paper, we extend this idea to semi-algebraic hypergraphs. For each $k\geq 2$, we find an upper bound on the number of hyperedges in a $k$-uniform $k$-partite semi-algebraic hypergraph without $K_{u_1,\dots,u_k}$ for fixed positive integers $u_1,\dots, u_k$. When $k=2$, this bound matches the one of Fox et.al. and when $k=3$, it is 
$$O\left((mnp)^{\frac{2d}{2d+1}+\varepsilon}+m(np)^{\frac{d}{d+1}+\varepsilon}+n(mp)^{\frac{d}{d+1}+\varepsilon}+p(mn)^{\frac{d}{d+1}+\varepsilon}+mn+np+pm\right),$$ where $m,n,p$ are the sizes of the parts of the tripartite hypergraph and $\varepsilon$ is an arbitrarily small positive constant. We then present applications of this result to a variant of the unit area problem,  the unit minor problem and intersection hypergraphs.

\end{abstract}

\section{Introduction}

The problem of Zarankiewicz \cite{Zarankiewcz} is a central problem in graph theory. It asks for the largest possible number of edges in an $m\times n$  bipartite graph that avoids $K_{u,u}$ for some fixed positive integer $u$. Here $K_{u,u}$ denotes the complete bipartite graph of size $u\times u$, and we say a graph $G$ avoids $H$ or $G$ is $H$-free if $G$ does not contain any subgraph congruent to $H$.  In 1954, K\H{o}v\'ari, S\'os and Tur\'an proved a general upper bound of form $O_u(mn^{1-1/u}+n)$, which is only known to be tight for $u=2$ and $u=3$. 

A natural context in which such bipartite graphs emerge is in incidence geometry, where edges represent incidences between two families of geometric objects, such as points and lines. Better bounds are known in certain of these cases, i.e. points and lines in $\R^2$ (Szemer\'edi-Trotter theorem \cite{S-T}), points and curves in $\R^2$ (Pach-Sharir \cite{Pach-Sharir}), and points and hyperplanes in $\R^d$ (Apfelbaum- Sharir \cite{Ap-Sharir}). Recently, Fox, Pach, Sheffer, Suk and Zahl \cite{Fox} generalized these results to all \emph{semi-algebraic graphs} (defined below).

Fixing some positive integers $d_1,d_2$, let $G=(P,Q,\mathcal{E})$ be a bipartite graph on sets $P$ and $Q$, where we think of $P$ as a set of $n$ points in $\R^{d_1}$ and $Q$ as a set of $m$ points in $\R^{d_2}$. We say $G$ is \emph{semi-algebraic with description complexity $t$} if there are $t$ polynomials $f_1,\dots, f_t\in \R[x_1,\dots, x_{d_1+d_2}]$, each of degree at most $t$ and a Boolean function $\Phi(X_1,\dots, X_t)$ such that for any $p\in P, q\in Q$: 
$$ (p,q)\in \mathcal{E} \iff \Phi(f_1(p,q)\geq 0,\dots, f_t(p,q)\geq 0)=1.$$
In other words, we can describe the incidence relation using at most $t$ inequalities involving polynomials of degree at most $t$. We now restate more formally the main result of \cite{Fox}: 

\begin{thm}[Fox, Pach, Sheffer, Suk and Zahl 2015 \cite{Fox}]\label{Fox}
Given a bipartite semi-algebraic graph $G=(P,Q,\mathcal{E})$ with description complexity $t$, if $G$ avoids $K_{u,u}$ then for any $\varepsilon>0$, $$|\mathcal{E}(G)|= O_{t,d_1,d_2,u,\varepsilon}\left(m^{\frac{d_1d_2-d_2}{d_1d_2-1}+\varepsilon}n^{\frac{d_1d_2-d_1}{d_1d_2-1}}+m+n\right).$$ 
When $d_1=d_2=2$ we can delete the $\varepsilon$ term. Moreover, if $P$ belongs to an irreducible variety of degree $D$ and dimension $e_1$ (with $e_1\leq d_1$) then 
$$|\mathcal{E}(G)|= O_{t,d_1,d_2,e_1, D, k,\varepsilon}\left(m^{\frac{e_1d_2-d_2}{e_1d_2-1}+\varepsilon}n^{\frac{e_1d_2-e_1}{e_1d_2-1}}+m+n\right).$$
\end{thm}

As many incidence graphs are semi-algebraic, this theorem and its proof method not only imply the incidence bounds mentioned previously (modulo the extra $\varepsilon$), but also imply many new ones. 
It is natural to ask if similar results hold for semi-algebraic hypergraphs.

\subsection*{Semi-algebraic hypergraphs}
Fix some integer $k\geq 2$. A hypergraph $H$ is called $k-$uniform  if each hyperedge is a $k$-tuple of its vertices. It is $k$-partite if its vertices can be partitioned into $k$ disjoint subset $P_1,\dots, P_k$ and each hyperedge is some tuple $(p_1,\dots, p_k)$ where $p_i\in P_i$ for $i=1,\dots, k$. We usually use $\mathcal{E}$, or $\mathcal{E}(H)$ to denote the set of hyperedges of $H$. 

Fix some positive integers $d_1,\dots, d_k$ and $t$. Let $H$ be a $k$-uniform $k$-partite hypergraph $H=(P_1,\dots, P_k, \mathcal{E})$ where $P_i$ is a set of $n_i$ points in $\R^{d_i}$ for $i=1,\dots, k$ and $\mathcal{E}$ is the set of all hyperedges. This hypergraph is said to be \emph{semi-algebraic with description complexity $t$} if there are $t$ polynomials $f_1,\dots,f_t\in \R[x_1,\dots, x_{d_1+\dots+d_k}]$, each of degree at most $t$, and a Boolean function $\Phi(X_1,\dots, X_t)$ such that for any $p_i\in P_i$, $i=1\dots, k$:
$$(p_1,\dots, p_k)\in \mathcal{E} \iff \Phi(f_1(p_1,\dots, p_k)\geq 0,\dots, f_t(p_1,\dots, p_k)\geq 0)=1.$$

Semi-algebraic hypergraphs have been the subject of much recent work (see for example \cite{Ramsey hypergraph, Semi-hypergraph regularity lemma2,Semi-hypergraph regularity lemma}), the main theme of which is that many classical theorems about hypergraphs (such as Ramsey's theorem and Szemer\'edi's regularity lemma) can be significantly improved in the semi-algebraic setting.
Since the graphs and hypergraphs arising in discrete geometry problems are often semi-algebraic (with low description complexity), such improved results have many applications there. 
Our paper follows this paradigm, improving upon a result of Erd\H{o}s regarding Zarankiewicz's problem for semi-algebraic hypergraphs. 

We first recall the classical result. Given positive integers $u_1,\dots, u_k$, let $K_{u_1,\dots, u_k}$ denote the complete $k-$uniform $k$-partite hypergraph $(U_1,\dots, U_k, \mathcal{E})$ where $|U_i|=u_i$ and any $k$-tuple in $U_1\times \dots\times U_k$ is a hyperedge. Zarankiewicz's problem for hypergraphs asks for the maximum number of hyperedges in a $k$-uniform hypergraph that does not contain a copy of $K_{u_1,\dots, u_k}$.
  The first statement in the following theorem was proved by Erd\H{o}s in \cite{Erdos}. Using his proof method, we obtain the second statement whose proof can be found in appendix \ref{appendix: Erdos}.

\begin{thm}[Erd\H{o}s 1964 \cite{Erdos}]\label{Erdos}  A $k$-uniform $K_{u,\dots, u}$-free hypergraph $H$ on $n$ vertices has at most $O_{u,k} (n^{k-\frac{1}{u^{k-1}}})$ hyperedges. 
More generally, if the $k$-partite $k$-uniform hypergraph $H=(P_1,\dots, P_k, \mathcal{E})$  is  $K_{u_1,\dots, u_k}$-free, then 
$|\mathcal{E}|=O_{u_1,\dots, u_k, k}\left(\left(n_k^{-1/u_1\dots u_{k-1}}+n_1^{-1}+\dots+ n_{k-1}^{-1}\right)\prod_{i=1}^k n_i \right)$,
where $n_i=|P_i|$ for $i=1,\dots, k$.
\end{thm}

\begin{remark}In the preliminary report on arXiv of \cite{Fox}, Fox et.al. improved this bound for semi-algebraic hypergraphs (Corollary 6.11 therein), but their proof is flawed: they claimed a $K_{u,\dots, u}$-free semi-algebraic hypergraph  $H=(P_1,\dots, P_k,\mathcal{E})$ is a semi-algebraic $K_{s,s}$-free bipartite graph between $P=\cup_{i\in S_1} P_i$ and $Q=\cup_{j\in S_2} P_j$ for any partition $S_1\cup S_2=[k]$ and some $s$ that only depends on $u,k$. It is true that this new graph is semi-algebraic with bounded complexity; however, it may not be  $K_{s,s}$-free for any fixed $s$. For example, the unit minor hypergraph in $\R^d$ on $n$ vertices (see \ref{sec:unit minor}) does not contain $K_{2,\dots, 2}$ but may contain $K_{1,(n-1)/d,\dots, (n-1)/d}$. 
\end{remark}

In this paper we present a way to extend theorem \ref{Fox} to semi-algebraic hypergraphs. Ultimately we will prove the number of hyperedges is bounded by a function of $n_1,\dots, n_k$, with the exponents depending on $d_i$, in this way resembling Theorem \ref{Fox}. However, the formulas involved are sufficiently complicated that we need to fix some notation before stating them precisely.

\subsection*{Notation}
Let $\vec{d}=(d_1,\dots, d_k)$ and $\vec{n}=(n_1,\dots, n_k)$ be  vectors in $\mathbb{Z}_+^k$.
For each $\vec{d}$, define a function $E_{\vec{d}}(\vec{n}):\R^k\to R$ via:
\begin{equation}\label{defi_E}
E_{\vec{d}}(\vec{n}):=E_{d_1,\dots, d_k}(n_1,\dots, n_k):= \prod_{i=1}^k n_i^{1-\frac{\frac{1}{d_i-1}}{k-1+\frac{1}{d_1-1}+\dots+\frac{1}{d_k-1}}}.
\end{equation}
For example, $E_{d}(n)=1$ for all $d$ and $n$, and $E_{d_1, d_2}(m,n) =m^{1-\frac{d_2-1}{d_1d_2-1}}n^{1-\frac{d_1-1}{d_1d_2-1}}$. \footnote{Note that here we do not require $d_i\geq  2$ because we can multiply the numerator and denominator of the exponent by $\prod_{i=1}^k (d_i-1)$ to get rid of the term $1/(d_i-1)$.}  
The function $E_{\vec{d}}(\vec{n})$  satisfies many nice properties that are discussed in subsection \ref{sec:property_EF}.
Let $[k]$ denote the set $\{1,\dots,k\}$, and for a subset $I\subset [k]$ let $\vec{d_I}$ denote the vector $(d_i)_{i\in I}\in \R^{|I|}$, and similarly let $\vec{n_I}=(n_i)_{i\in I}$. 
For $i\in [k]$, let $\pi_i$ be the projection of $\R^k$ to $\langle \frak{e}_i \rangle^{\perp}$; i.e. for any vector $\vec{a}\in \R^k$, $\pi_i(\vec{a})=(a_1,\dots,a_{i-1},a_{i+1},\dots,a_k)$. 
For each $\varepsilon>0$ and each $\vec{d}$,  define a function $F^\varepsilon_{\vec{d}}:\R^k\to \R$ as follows:
\begin{equation}\label{defi_F}
F^\varepsilon_{\vec{d}}(\vec{n}):=\sum_{I\subset [k], |I|\geq 2}  E_{\vec{d_I}}(\vec{n_I})\prod_{i\in I} n_i^\varepsilon \prod_{i\notin I} n_i  + \left(\frac{1}{n_1}+\dots+\frac{1}{n_k}\right)\prod_{i=1}^k n_i
\end{equation}

\subsection*{Our results} 
We first prove a more general version of Theorem \ref{Fox} for semi-algebraic graphs, which shows that, if \emph{both} $P$ and $Q$ belong to irreducible varieties of dimensions $(e_1, e_2)$ inside of $\R^{d_1}$ and $\R^{d_2}$, then we may replace both dimensions $(d_1, d_2)$ by $(e_1, e_2)$ in the upper bound. 
\begin{thm}\label{Fox-extension}
Given a semi-algebraic bipartite graph $G=(P, Q,\mathcal{E})$ where $P$ is a set of $m$ points in an irreducible variety of dimension $e_1$ and complexity at most $D$ in $\R^{d_1}$ and $Q$ is a set of $n$ points in an irreducible variety of dimension $e_2$ and complexity at most $D$ in $\R^{d_2}$. If $G$ has description complexity $t$ and contains no $K_{u,u}$, then $$|\mathcal{E}| =O_{e_1,e_2,D,t,u,\varepsilon}(m^{\frac{e_1e_2-e_2}{e_1e_2-1}+\varepsilon}n^{\frac{e_1e_2-e_1}{e_1e_2-1}+\varepsilon}+m+n).$$
\end{thm}
We then extend this result to $k$-uniform semi-algebraic hypergraphs for any $k\geq 3$.  

\begin{thm}[Main Theorem]\label{main thm}
Given a $k$-uniform $k$-partite hypergraph $H=(P_1,\dots, P_k,\mathcal{E})$ with description complexity $t$ which avoids $K_{u,\dots,u}$, then 
\begin{equation}\label{main ineq}
|\mathcal{E}(H)|=O_{t,k,u, \vec{d}, \varepsilon}\left(F^\varepsilon_{\vec{d}}(\vec{n})\right).
\end{equation}
Moreover, if for each $i\leq k$, $P_i$ belongs to an irreducible variety of degree $D$ and dimension $e_i\leq d_i$, then 
$|\mathcal{E}(H)|=O_{t,k,u, \vec{d}, D, \varepsilon}\left(F^\varepsilon_{\vec{e}}(\vec{n})\right)$
where $\vec{e}=(e_1,\dots, e_{k})$.
\end{thm}

\begin{remark}
\begin{itemize}
\item[(i)]\sloppy Without the $\varepsilon$ term, function $F$ has a nicer form $F_{\vec{d}}(\vec{n})=\sum_{\emptyset\neq I\subset[k]} E_{\vec{d_I}}(\vec{n_I})\prod_{i\notin I}n_i$. 
As mentioned in \cite{Fox}  the term $n_1^\varepsilon$ is not necessary when $k=d_1=d_2=2$, and we conjecture that this artifact of the proof can be removed in general.

\item[(ii)] When $k=2$, this theorem implies  $|\mathcal{E}(H)|\lesssim F^{\varepsilon}_{d_1, d_2}(m,n) = m^{\frac{d_1d_2-d_2}{d_1d_2-1}+\varepsilon}n^{\frac{d_1d_2-d_1}{d_1d_2-1}+\varepsilon}+m+n$. It is slightly weaker yet essentially the same as the bound in Theorem \ref{Fox}. 
Indeed, as we shall see in remark \ref{assumption}, the term $m+n$ dominates unless $n^{1/d_2}\leq m\leq n^{d_1}$. Hence $$m^{\frac{d_1d_2-d_2}{d_1d_2-1}+\varepsilon}n^{\frac{d_1d_2-d_1}{d_1d_2-1}+\varepsilon}+m+n \lesssim m^{\frac{d_1d_2-d_2}{d_1d_2-1}+\varepsilon'}n^{\frac{d_1d_2-d_1}{d_1d_2-1}} + m +n,$$ where $\varepsilon'=(d_2+1)\varepsilon$. In general, we can prove a stronger result where $\varepsilon$ appears only once in each term of $F^\varepsilon_{\vec{d}}(\vec{n})$.

\item[(iii)] When $d_1=d_2=\dots=d$, we have
$$|\mathcal{E}(H)|=O_{t,k,u,\vec{d},\varepsilon}\left(\sum_{j=2}^k \sum_{I\subset[k];|I|=j} \left(\prod_{i\notin I} n_i\right)\left(\prod_{i\in I}n_i\right)^{1-\frac{1}{(j-1)d+1}+\varepsilon}+\left(\sum_{i=1}^k\frac{1}{n_i}\right)\prod_{i=1}^k n_i\right).
$$
When $k=3$ we get the formula mentioned in the abstract.
Assume furthermore $n_1=\dots=n_k=n$ the bound becomes $n^{k-\frac{k}{(k-1)d+1}+\varepsilon}$ which is smaller than $n^{k-\frac{1}{u^{k-1}}}$ (the bound in Theorem \ref{Erdos}) when $d<\frac{ku^{k-1}-1}{k-1}$. 
\item[(iv)] If, on the other hand, if $u^{k-1}<d_i$ for some $i$, say $u^{k-1}<d_k$, we can use Theorem \ref{Erdos} instead of Proposition \ref{prop-init-bound} in the proof to derive the bound $F^\varepsilon_{d_1,\dots, d_{k-1}, u^{k-1}}(\vec{n})$ where we replace $d_k$ by $u^{k-1}$.
\end{itemize}
\end{remark}

\subsection*{Applications} Our main result, Theorem \ref{main thm}, implies nontrivial  bounds for many geometric problems.   In section \ref{section:application}, we present several applications including a variant of the unit area problem, the unit minor problem, and intersection hypergraphs.

First we find an upper bound $O_\varepsilon(n^{12/5+\varepsilon})$ for the number of triangles with area very close to 1, say between 0.9 and 1.1, formed by $n$ points in the plane, assuming for some fixed $u>0$ there does not exist $3u$  points $a_i, b_i, c_i$, $i\in[u]$ among those given points such that the triangles formed by $(a_i, b_j, c_k)$ have area between 0.9 and 1.1 for any $i,j,k\in[u]$. 

The unit minor problem asks for the largest number of unit $d\times d$ minors in a $d\times n$ matrix with no repeated columns. This problem was considered in \cite{minor in matrix} but only for the case the matrix is totally positive, which is a much stronger assumption. In \ref{sec:unit minor} 
we prove an upper bound $O_{d,\varepsilon} (n^{d-\frac{d}{d^2-d+1}+\varepsilon})$ for the number of unit minors for any matrix with no repeated columns. As a corollary, the maximum number of unit volume $d$-simplices formed by $n$ points in $\R^d$ is $O_{d,\varepsilon}(n^{d+1-\frac{d}{d^2-d+1}+\varepsilon})$. 

The last application is about intersection hypergraphs.
Given a set $S$ of geometric objects, their intersection graph $H(S)$ is defined as a graph on the vertex set $S$, in which two vertices are joined by an edge if and only if the corresponding elements of $S$ have a point in common.  Fox and Pach proved that if $H(S)$ is $K_{u,u}$-free for some $u>0$, then $H(S)$ has $O(n)$ edges when $S$ is a set of $n$ line segments \cite{Fox-intersection graph} or arbitrary continuous arcs \cite{Fox-intersection graph2, Fox-intersection graph3} in $\R^2$ , 
Mustafa and Pach \cite{intersection-hypergraph} gave a hypergraph version of this, but only for simplices. Given a set $S$ of geometric objects (usually of dimension $d-1$) in $\R^d$, their intersection hypergraph $H(S)$ is defined as a $d$-uniform hypergraph on the vertex set of $S$, in which $d$ vertices form a hyperedge if and only if the corresponding sets in $S$ have a point in common. Mustafa and Pach \cite{intersection-hypergraph} proved that  if $H(S)$ is $K_{u,\dots, u}$-free then $H(S)$ has $O_{d,\epsilon}(|S|^{d-1+\epsilon})$ hyperedges given $S$ is a set of $(d-1)$-dim simplices in $\R^d$. 
In this paper, we found a nontrivial upper bound for many other types of geometric objects such as spheres. 

\subsection*{Organization} In section \ref{sec:preliminary}, we introduce several useful tools such as the Milnor-Thom's theorem, the polynomial partitioning method and a packing-type result from set system theory. We then prove Theorem \ref{Fox-extension} in section \ref{sec:general Fox} and our main theorem in section \ref{sec:general_k}. Section \ref{section:application} is devoted to applications. We end with several open problems in section \ref{section:discussion}.

\subsection*{Acknowledgements} The author is immensely grateful to Larry Guth for suggesting this problem and for his  help throughout the project. She thanks Ben Yang for providing the short and beautiful proof of lemma \ref{no K222}, as well as Josh Zahl and Andrew Suk for helpful conversations. She also thanks Ethan Jaffe, Malcah Effron and Jake Wellens for proofreading this paper.
Finally, she would like to thank the referees for many helpful suggestions.
\section{Preliminary}\label{sec:preliminary}
\subsection{Milnor-Thom type results}

 Milnor-Thom's theorem \cite{Milnor, Thom} states that the zero set  of a degree $D$ polynomial $f$, denoted by $Z(f)$, divides $\R^d$  into at most $(50D)^d$ connected components (i.e. $\R^d\setminus Z(f)$ has at most $(50D)^d$ connected components). Basu, Pollack and Roy extended this result to the case when we restrict our attention to a variety inside $\R^d$. 
 
 A \emph{sign pattern} for a set of $s$ $d$-variate polynomials $\{f_1,\dots, f_s\}$ is a vector $\sigma\in \{-1,0,+1\}^s$. A sign pattern $\sigma$ is \emph{realizable} over a variety $V\subset\R^d$ if there is some $x\in V$ such that (sign$(f_1(x))$, sign$(f_2(x)),\dots,$sign$(f_s(x)))=\sigma$. The set of all such $x$ is the realization space of $\sigma$ in $V$, denoted by $\Omega_\sigma$.
\begin{thm}[Basu, Pollack and Roy, 1996 \cite{Basu-Pollack-Roy}]\label{Basu-Pollack-Roy}
Given positive integers $d, e, M, t, l$, let $V$ be an $e$-dimensional real algebraic set in $\R^d$ of complexity\footnote{A variety has \emph{complexity} at most $M$ if it can be realized as the intersection of zero-sets of at most $M$ polynomials, each of degree at most $M$.} at most $M$, 
and let $f_1,\dots, f_s$ be $d$-variate real polynomials of degree at most $t$. Then the total number of connected components of $\Omega_\sigma$ for all realizable sign patterns $\sigma$ of $\{f_1,\dots, f_s\}$ is at most $O_{M,d,e}((ts)^e)$. 
\end{thm}
Milnor-Thom's theorem follows from this result by taking $s=1, V=\R^d$ and noting that $\R^d\setminus Z(f)$ is the union of two realizable spaces $\{f>0\}$ and $\{f<0\}$. This result  implies if we restrict to a variety $V$ with dimension $e$ and bounded complexity in $\R^d$, then the number of connected components that $f_1,\dots, f_s$ partition $V$ grows with $e$ instead of $d$. Since each realizable sign pattern has at least one connected component, we get a bound on the number of  {realizable sign patterns}. 
\begin{cor}\label{Basu-cor}
Under the same assumption as above, the number of realizable sign patterns of $(f_1,\dots, f_s)$ in $V$ is at most $O_{M,d,e}((ts)^e)$.
\end{cor}
Furthermore, a similar result holds if we replace $V$ by $V\setminus W$ for some variety $W$ with bounded complexity.
\begin{thm}[Theorem A.2 in \cite{Tao}] 
\label{Tao}
Given positive integers $d, e, M, t$ such that $e\leq d$, let $V$ and $W$ be a real algebraic sets in $\R^d$ of complexity at most $M$ such that $V$ is $e$-dimensional. Then for any polynomial $P:\R^d\to \R$ of degree $t \geq 1$, the set $\{x\in V\setminus W: P(x)\neq 0\}$ has $O_{M,d,e}(t^e)$ connected components.
\end{thm}

\subsection{Polynomial partitioning}\label{sec:poly partitioning}
The polynomial partitioning method was first introduced by Guth and Katz in \cite{Larry} in 2010, and numerous modifications have appeared since then. In this paper we use the version proved in \cite{Fox}. 
Given a set $P$ of points in $\R^d$, we say a polynomial $f\in\R[x_1,\dots,x_d]$ is an \emph{$r$-partitioning for $P$} if the zero set of $f$ (denoted $Z(f)$) divides the space into open connected components, each of which contains at most $|P|/r$ points of $P$.

\begin{thm}[Theorem 4.2 in \cite{Fox}]\label{poly partitioning}
Let $P$ be a set of points in $\R^d$, and let $V\subset \R^d$ be an irreducible variety of degree $D$ and dimension $e$. Then for big enough $r$, there exists an $r-$partitioning polynomial $g$ for $P$ such that $g\notin I(V)$ and $\deg g\leq C_{part} \cdot r^{1/e}$ where $C_{part}$ depends only on $d$ and $D$. 
\end{thm}
This theorem implies in $\R^d$, if we restrict our attention to points in an irreducible variety of small degree and dimension $e<d$, then we can perform a polynomial partitioning the same way as in $\R^{e}$.

\subsection{A result about set systems}\label{sec:VC-dim}

A \emph{set system} $\F$ over a ground set $P$ is just a collection of subsets of $P$ (here we allow $\F$ to contain repeated elements).
Given a bipartite graph $G=(P,Q,\mathcal{E})$, let $N(p):=\{q\in Q: (p,q)\in \mathcal{E}\}$ denote the neighbors of a vertex $p\in P$, and likewise for $q \in Q$. Then $\mathcal{F}_1:=\{N(q): q\in Q\}$, and $\mathcal{F}_2:=\{N(p):p\in P\}$ are two set systems with ground sets $P$ and $Q$ respectively. 
The \emph{primal shatter function} of a set system $(\mathcal{F},P)$ is defined as
$$\pi_\mathcal{F}(z)=\max_{P'\subset P, |P'|=z}|\{A\cap P': A\in\mathcal{F}\}|.$$

Given two sets $A$ and $B$, we say $A$ \emph{crosses} $B$ if $A\cap B\notin\{\emptyset, B\}$. The following lemma is essential to our proof.

\begin{lem}[Observation 2.6 in \cite{Fox}]\label{lem packing}
For the set systems $(\F_1, P)$ and $(\mathcal{F}_2,Q)$ defined from the graph $G=(P,Q)$, if $\pi_{\mathcal{F}_1}(z)\leq c z^{d}$ for all $z$, then for each $u$, there exists $u$ points $q_1,\dots, q_u\in Q$ such that at most $O_{u,d,t,c}(|P||Q|^{-1/d})$ sets from $\mathcal{F}_2$ cross $\{q_1,\dots, q_u\}$. 
\end{lem}
In the paper we use this result as a black box. 
For readers who are interested in some intuition, it follows from a packing-type result in $VC$ dimension theory. The  \emph{Vapnik-Chervonenkis (VC) dimension} of a set system $\mathcal{F}$ is the largest $D$ such that $\pi_F(D)=2^D$. It is easy to see if $\pi_\mathcal{F}(z)\leq cz^d$ for some fixed $c,d>0$ then the VC dimension of $\mathcal{F}$ does not excess $c'd\log d$ for some $c'>0$.  Intuitively, if a set system $F$ has a bounded VC dimension,  and its elements are well separated in the symmetric distance, then $F$ cannot have too many elements; it is analogous to how we pack spheres in Euclidean spaces. To be precise, 
lemma 2.5  in \cite{Fox} says that if a set system $\mathcal{F}$  on a ground set of $m$ elements satisfies $\pi_\mathcal{F}(z)\leq cz^d$ for all $z$ and $|(A_1\cup\dots\cup A_u)\setminus (A_1\cap \dots\cap A_u)|\geq \delta$ for all choices of $A_1,\dots, A_u\in \F$ and some fixed $u,\delta$, then $|\mathcal{F}|\leq C_{pack}  (m/\delta)^d$ for some constant $C_{pack}$ that depends on $c,d,u$. 
 In lemma \ref{lem packing}, assume  for each set $\{q_1,\dots, q_u\}$ in $Q$ there are more than $\delta:=c_1|P||Q|^{-1/d}$ sets from $\mathbb{F}_2$ crossing it. This implies $|(N(q_1)\cup \dots\cup  N(q_u)\setminus (N(q_1)\cap\dots\cap N(q_u)|\geq \delta $ for all choices of $q_1,\dots, q_u\in Q$. In other words, $\mathcal{F}_1$ is $(u,\delta)$ separated. Applying the packing lemma to $\F_1$ leads to a contradiction for small enough $c_1$.
 $$|Q|=|\mathcal{F}_1|\leq c_1 \left(\frac{|P|}{\delta}\right)^d=c_1\left(\frac{|P|}{C_{pack}|P||Q|^{-1/d}}\right)^d=c_1C_{pack}^d|Q|<|Q|$$

\subsection{Some properties of functions $E_{\vec{d}}$ and $F^\varepsilon_{\vec{d}}$}\label{sec:property_EF} 
In this subsection, we collect some useful properties of $E_{\vec{d}}$ and $F^\varepsilon_{\vec{d}}$. All the proofs are quite straightforward and can be found in appendix \ref{appendix: property E}.
Recall $E_{\vec{d}}(\vec{n})=\prod_{i=1}^k n_i^{\alpha_i}$ where  $\alpha_i=1-\frac{1/(d_i-1)}{k-1+\sum_l 1/(d_l-1)}$.

\begin{lem}\label{proof-matrix-equ} 
For each $i\in[k]$ we have
$\alpha_i= \sum_{j\neq i} d_j(1-\alpha_j)$. In other words, the exponents $\{\alpha_i\}$ satisfy a nice system of equations:
\begin{equation}\label{matrix}
\begin{pmatrix}
1 & d_2 & \dots & d_k\\
d_1 & 1 &\dots & d_k\\
\vdots & \vdots & \ddots &\vdots\\
d_1 & d_2 & \dots & 1
\end{pmatrix}
\begin{pmatrix}
\alpha_1\\
\alpha_2\\
\vdots\\
\alpha_k
\end{pmatrix}=
\begin{pmatrix}
\sum_{i=1}^k d_i- d_1\\
\sum_{i=1}^k d_i- d_2\\
\vdots\\
\sum_{i=1}^k d_i- d_k
\end{pmatrix}
\end{equation}

\end{lem}

\begin{cor}\label{inflate by r property}
 For any $r>0$ and each $i\in[k]$ we have 
 $$r^{d_1+\dots+d_{k-1}}E_{\vec{d}}(\frac{n_1}{r^{d_1}},\dots,\frac{n_{k-1}}{r^{d_{k-1}}},\frac{n_k}{r})=E_{\vec{d}}(\vec{n}).$$

Similar equalities, in which we replace the special index $k$ by some other index, also hold.
\end{cor}
\begin{proof}
\begin{align*}
LHS&=r^{d_1+\dots+d_{k-1}} \left(\frac{n_1}{r^{d_1}}\right)^{\alpha_1}\dots \left(\frac{n_{k-1}}{r^{d_{k-1}}}\right)^{\alpha_{k-1}}\left(\frac{n_k}{r}\right)^{\alpha_k}\\
&=\left(\prod_i n_i^{\alpha_i}\right) r^{d_1+\dots+d_{k-1}-\sum_{j=1}^{k-1} d_j\alpha_j-\alpha_k}= E_{\vec{d}}(\vec{n})
\end{align*}
In the last step, the exponent of $r$ becomes 0 because $\sum_{j=1}^{k-1} d_j(1-\alpha_j)= \alpha_k$ by Lemma \ref{proof-matrix-equ}.
\end{proof}

\begin{lem}\label{compare d and d-1}
Let $\mathfrak{e}_1,\dots, \mathfrak{e}_k$ be the standard basis in $\R^k$.  Then $F^\varepsilon_{\vec{d}-\mathfrak{e}_i}(\vec{n})\leq F^\varepsilon_{\vec{d}}(\vec{n})$ assuming $n_i\geq n_j^{1/(d_j)}$ for any $j\neq i$.
\end{lem}

\begin{lem}\label{E_d dominant}

 Assume for each $i\in[k]$ we have 
\begin{equation}\label{assume VC bound dominate}
n_i^{-1/d_i}\prod_{j=1}^k n_j\geq n_iF^\varepsilon_{\pi_i(\vec{d})}(\pi_i(\vec{n})),
\end{equation}
then
$ E_{\vec{d}}(\vec{n})\prod_{i=1}^k n_i^\varepsilon \geq cF^\varepsilon_{\vec{d}}(\vec{n})$ for some constant $c$. In other words, $E_{\vec{d}}(\vec{n})\prod_i n_i^\varepsilon$ is the dominant term of $F^\varepsilon_{\vec{d}}(\vec{n})$.
\end{lem}

\section{Proof of theorem \ref{Fox-extension}}
\label{sec:general Fox}
In this section, we first outline the proof of the last statement in theorem \ref{Fox} and then carry out the modifications needed to prove theorem \ref{Fox-extension}. In both cases, $G=(P,Q)$ is a semi-algebraic $K_{u,u}$-free graph with description complexity $t$. Recall in the last statement of theorem \ref{Fox},  we assume $P$ is a set of $m$ points from an irreducible variety with dimension $e_1$ and bounded complexity in $\R^{d_1}$ and $Q$ is a set of $n$ points in $ \R^{d_2}$, and the conclusion is 
\begin{equation}\label{fox conclusion}|\E|\lesssim m^{\frac{e_1d_2-d_2}{e_1d_2-1}+\varepsilon}n^{\frac{e_1d_2-e_1}{e_1d_2-1}}+m+n. \end{equation}

In theorem \ref{Fox-extension}, we assume additionally that $Q$ belongs to an irreducible variety with bounded complexity and dimension $e_2$ in $\R^{d_2}$ and wish to prove 
\begin{equation}\label{extension conclusion} |\E|\lesssim m^{\frac{e_1e_2-e_2}{e_1e_2-1}+\varepsilon}n^{\frac{e_1e_2-e_1}{e_1e_2-1}}+m+n,\end{equation}
i.e. that we can replace $d_2$ by $e_2$. Note that this is similar in spirit to theorem \ref{Basu-Pollack-Roy}, which is precisely what we shall use.

Fox, Pach, Sheffer, Sulk and Zahl's proof of theorem \ref{Fox} proceeds in two steps: 
first, lemma \ref{lem packing} is used to show $|\E(G)|\lesssim mn^{1-1/d_2}+n$, and then induction and polynomial partitioning are used to derive the desired bound. 

The first step begins by proving that $$\pi_{\mathcal{F}_1}(z)\leq c(t,d_2) z^{d_2}$$ where the set systems $(\F_1,P)$ and $(\F_2,Q)$ are defined in  \ref{sec:VC-dim}. For each $p\in P$, its neighbors belong to a semi-algebraic set $\gamma_p$ in $\R^{d_2}$ defined as followed:
$$\gamma_{p}:=\{x\in\R^{d_2}:\Phi(f_1(p,x)\geq 0,\dots, f_t(p,x)\geq 0)=1\}.$$ 
For any $z$ points $p_1,\dots, p_z\in P$, their  semi-algebraic sets $\gamma_{p_1},\dots, \gamma_{p_z}$ are defined by at most $tz$ polynomials $\{f_j(p_i,x): i\in[z], j\in[t]\}$.  By the Milnor-Thom theorem, these $tz$ polynomials have at most  $c(t,d_2) z^{d_2}$ realizable sign patterns in $\R^{d_2}$. If two points $q$ and $q'$ in $Q$ share the same sign pattern, their neighborhoods in $P$ restricted to $\{p_1,\dots, p_z\}$ are the same. 
As a consequence, $\pi_{\mathcal{F}_1}(z)=\max_{p_i}|\{N(q)\cap\{p_1,\dots, p_z\}: q\in Q\}|\leq c(t,d_2) z^{d_2}$.

Applying lemma \ref{lem packing}, there exists $q_1,\dots, q_u\in Q$ such that at most $O(|P||Q|^{-1/d_2})=O(mn^{-1/d_2})$ sets from $\mathcal{F}_2$ cross $\{q_1,\dots, q_u\}$. On the other hand, there are at most $u-1$ sets from $\mathcal{F}_2$ that contain $\{q_1, \dots, q_u\}$ because the graph is $K_{u,u}$-free. Therefore, the degree of $q_1$ in $G$ is at most $O(mn^{-1/d_2})+(u-1)$. Removing $q_1$ and repeating this argument at most $n$ times, we have $|\E|\lesssim mn^{1-1/d_2}+n$. The term $n$ dominates unless $n<m^{d_2}$. Thus from now on we assume $n<m^{d_2}$. 

In the second step, we view edges of $G$ as incidences between $P$ and $n$ semi-algebraic sets $\{\gamma_q:q\in Q\}$ in $\R^{d_1}$ where $\gamma_q$ is the set of all potential neighbors of $q$ in $\R^{d_1}$, i.e. 

$$\gamma_{q}:=\{x\in\R^{d_1}:\Phi(f_1(x,q)\geq 0,\dots, f_t(x,q)\geq 0)=1\}.$$ 

We prove (\ref{fox conclusion}) by double induction -- first on $e_1$ and then on $m+n$. The result is obvious for $e_1=0$ (since zero-dimensional irreducible varieties are just singletons). For a fixed $e_1\geq 1$, the result is true for small $m+n$. In the induction step, by theorem \ref{poly partitioning}, for a parameter $r$ to be chosen later, we can find a polynomial $f$ not vanishing on $V$  of small degree to partition the points in $P$ equally with respect to the variety $V$ of dimension $e_1$. More precisely, we can find $f$ of degree at most $C_{part} r$ such that $V\setminus Z(f)$ has $s=O(r^{e_1})$ connected components, or \emph{cells}, such that each cell contains $O(\frac{m}{r^{e_1}})$ points of $P$. There are 3 types of incidences between a point $p$ and a semi-algebraic set $\gamma_q$, which we shall group into sets $I_1, I_2,$ and $I_3$ respectively: (1) where $p$ belongs to $Z(f)\cap V$; (2) where $p$ belongs to a cell in $V\setminus Z(f)$ and $\gamma_q$ contains the whole cell; and (3) where $p$ belongs to a cell and $\gamma_q$ \emph{crosses} the cell, i.e. has non-empty intersection with the cell but does not contain the entirety of it. As $|\mathcal{E}| = |I_1| + |I_2| + |I_3|$, it suffices to bound the number of each type of incidence by the right side of (\ref{fox conclusion}).

To bound $|I_1|$, note that $V\cap Z(f)$ is a variety of dimension at most $e_1-1$, so we can apply the inductive hypothesis and get $I_1\lesssim E_{e_1-1, d_2}(m,n)+m+n$. By lemma \ref{compare d and d-1}, $E_{e_1-1, d_2}(m,n)\lesssim E_{e_1,d_2}(m,n)$ when $m>n^{1/d_2}$, which we are allowed to assume (this was the point of step one!)

The number of incidences in $I_2$ is bounded by $um+unr^{e_1}$ because the graph is $K_{u,u}$-free. Indeed, for each cell, either the cell contains fewer than $u$ points, or it is contained in fewer than $u$ semi-algebraic sets $\gamma_q$. The contribution in the first case is at most $un$ in each cell and at most $unr^{e_1}$ in total. The contribution in the second case is at most $um$. We can choose $r$ small enough so that the $nr^{e_1}$ is less than  $m^{\frac{e_1d_2-d_2}{e_1d_2-1}+\varepsilon}n^{\frac{e_1d_2-e_1}{e_1d_2-1}}$ when $n<m^{d_2}$.

Finally, to bound $|I_3|$, suppose there are $s=O(r^{e_1})$ cells $\Omega_1,\dots, \Omega_s$, where each $\Omega_i$ contains $m_i=O(m/r^{e_1})$ points and is crossed by $n_i$ semi-algebraic sets. We claim that each semi-algebraic set $\gamma_q$ crosses  at most $O(r^{e_1-1})$ cells. Indeed, each $\gamma_q$ is defined by $t$ polynomials $f_1(x, q),\dots, f_t(x,q)$. In order for $\gamma_q$ to cross a cell in $V\setminus Z(f)$, some polynomial, say $f_1$, must not vanish on $V$. Then $Z(f_1)\cap V$ is some variety of dimension at most $e_1-1$. By theorem \ref{Tao}, $f$ partitions this variety in at most $O_{t,d_1, D} (r^{e_1-1})$ cells; this in turn implies $\gamma_q$ crosses $O(r^{e_1-1})$ cells. As a result, the total number of pairs $(\Omega_i, \gamma_q)$ such that $\gamma_q$ crosses $\Omega_i$ is $\sum_{i=1}^s n_i\lesssim r^{e_1-1} n$. We apply the inductive hypothesis in each cell (for smaller $m_i+n_i$), add them up, and use H\"older's inequality:
\begin{align*}
I_3& =\sum_{i=1}^s I(m_i, n_i)\lesssim \sum_{i=1}^s \left( m_i^{\frac{e_1d_2-d_2}{e_1d_2-1}+\varepsilon}n_i^{\frac{e_1d_2-e_1}{e_1d_2-1}}+m_i+n_i\right)\\
&\lesssim \sum_{i=1}^s \left(\frac{m}{r^{e_1}}\right)^{\frac{e_1d_2-d_2}{e_1d_2-1}+\varepsilon}n_i^{\frac{e_1d_2-e_1}{e_1d_2-1}}+ m+\sum_{i=1}^s n_i
\\
&\lesssim \left(\frac{m}{r^{e_1}}\right)^{\frac{e_1d_2-d_2}{e_1d_2-1}+\varepsilon} s\left(\frac{\sum n_i}{s}\right)^{\frac{e_1d_2-e_1}{e_1d_2-1}}+m+nr^{e_1-1}
\\
&\lesssim \left(\frac{m}{r^{e_1}}\right)^{\frac{e_1d_2-d_2}{e_1d_2-1}+\varepsilon} r^{e_1}\left(\frac{nr^{e_1-1}}{r^{e_1}}\right)^{\frac{e_1d_2-e_1}{e_1d_2-1}}+m+nr^{e_1-1}
\\
&=r^{-\varepsilon e_1} m^{\frac{e_1d_2-d_2}{e_1d_2-1}+\varepsilon}n^{\frac{e_1d_2-e_1}{e_1d_2-1}}+m+n r^{e_1-1}
\end{align*}
We choose $r$ small enough so that $nr^{e_1-1}$ is bounded by the first term when $n<m^{d_2}$, but also large enough so that $r^{-\varepsilon e_1}$ is strictly smaller than the coefficient chosen in \eqref{main ineq}.
This finishes the proof of theorem \ref{Fox}.

In theorem \ref{Fox-extension}, since  $Q$ belongs to a variety of dimension $e_2$ with complexity $D$,  by corollary \ref{Basu-cor}, $\pi_{\mathcal{F}_1}(z)\leq c(t,d_2, e_2, D) z^{e_2}$. This is ``step one" in the above proof outline, with $d_2$ replaced by $e_2$. 
Combining with lemma \ref{lem packing} we get:
\begin{lem} \label{cross} Given $G=(P,Q,\mathcal{E})$ as in the statement of theorem \ref{Fox-extension}. Then there exist $u$ points $q_1,\dots, q_u\in Q$ such that at most $O_{u,e_2, d_2, t, D}(|P||Q|^{-1/{e_2}})$ sets from $\mathcal{F}_2$ cross $\{q_1,\dots, q_u\}$.
\end{lem}
We can subsequently replace $d_2$ by $e_2$ in all other steps and get theorem \ref{Fox-extension}. 
\qed

\section{Proof of the main theorem}\label{sec:general_k}
\subsection*{Overview of the proof.} Our proof will proceed by induction on $k$ -- the statement clearly holds for $k=1$, so fix some $k\geq 2$. For the inductive step, we follow the same general strategy as in the previous section: first using lemma \ref{lem packing} to obtain a bound where the exponents only depend on $\vec{d}$, and then using polynomial partitioning to get a better bound. We highlight some differences: in the first step, we view the hypergraph as a semi-algebraic bipartite graph between $P_k$ and $P_1\times\dots\times P_{k-1}$ to apply lemma \ref{lem packing}.  To bound the number of edges that contain $u$ points in $P_k$, we need to use the induction assumption for some $(k-1)$-uniform semi-algebraic hypergraph on $P_1,\dots, P_{k-1}$. In the second step, we reduce the problem to counting incidences between $n_k$ semi-algebraic sets defined by $P_k$ and the grid $P_1\times \dots\times P_{k-1}$ in $\R^{d_1+\dots+d_{k-1}}$. If we simply apply the usual polynomial partitioning, each cell may not have the structure of a $k$-partite hypergraph. We overcome this by using a \emph{product} of $k-1$ polynomials $f_1,\dots, f_k$ of the same degree,  where $f_i$ partitions $P_i$ equally in $\R^{d_i}$ for  $i<k$.  
By doing this, we preserve the grid structure and thus can use induction on a smaller grid in each cell. 

We begin now with step one.
\begin{prop} \label{prop-init-bound}
Given a $k$-uniform $k$-partite semi-algebraic $K_{u,\dots, u}$-free hypergraph $(P_1,\dots, P_k,\E)$ with description complexity $t$ such that for each  $i\in [k]$, $A_i$ belongs to $V_i$, an irreducible variety in $\R^{d_i}$ of dimension $e_i\leq d_i$ and degree bounded by $D$, then 
$$|\mathcal{E}(H)|=O_{\vec{d},k,u,D,\varepsilon}\left( n_1\dots n_{k-1}n_k^{1-1/e_k}+ n_k F^\varepsilon_{e_1,\dots, e_{k-1}}(n_1,\dots, n_{k-1})\right).$$
\end{prop}
\begin{proof}
Given $p_i\in P_i$ for $i\in[k-1]$, let $N(p_1,\dots, p_{k-1}):=\{p_k\in P_k: (p_1,\dots, p_k)\in \E\}$ be their neighbors in $P_k$. Moreover, let $\mathcal{F}:=\{N(p_1,\dots, p_{k-1}): p_i\in P_i\}$ be the set system on the ground set $P_k$.

We first claim there exist $u$ points $q_1,\dots, q_u$ in $P_k$ such that the number of sets in $\F$ that cross $\{q_1,\dots, q_u\}$ is at most $c_1n_1\dots n_{k-1}n_k^{-1/e_k}$ for some $c_1(\vec{d}, t,u)$. 
Indeed, we can think of our hypergraph as a bipartite graph $(P_1\times \dots\times P_{k-1}, P_k)$ where there is an edge between $(p_1,\dots, p_{k-1})\in P_1\times\dots\times P_{k-1}$ and $p_k\in P_k$ iff there is a hyperedge $(p_1,\dots, p_k)$ in $\mathcal{E}$. Clearly this bipartite graph is semi-algebraic with description complexity $t$, hence we can apply lemma \ref{cross} for this graph where $Q$ is $P_k$ and obtain the desired claim.

 Next, we claim that the number of sets in $\mathcal{F}$ that contain $\{q_1,\dots, q_u\}$ is at most  $c_2F^\varepsilon_{d_1,\dots, d_{k-1}}(n_1,\dots, n_{k-1})\allowbreak$ for some constant $c_2$. Indeed,
let $\E'$ be the sets of all $(p_1,\dots, p_{k-1})\in P_1\times \dots\times P_{k-1}$ such that their neighbor set $N(p_1,\dots,p_{k-1})$ contains $\{q_1,\dots, q_u\}$. Then the \emph{induced} $(k-1)$-uniform hypergraph $H'=(P_1,\dots,P_{k-1},\mathcal{E}')$ is semi-algebraic with description complexity $tu$ and contains no $K_{u,\dots,u}$. 
Inductively,  $|\mathcal{F}'|\leq  c_2F^\varepsilon_{e,\dots, e_{k-1}}(n_1,\dots, n_{k-1})$. 

Combining these two claims,  we conclude that there are at most $(c_1+c_2)(n_1\dots n_{k-1}n_k^{-1/e_k}+  F^\varepsilon_{e_1,\dots, e_{k-1}}(n_1,\dots, n_{k-1}))$ hyperedges in $H$ that contain $q_1$ (because each such hyperedge must either cross or contain $\{q_1,\dots, q_u\})$. Removing this vertex and repeating this argument until there are fewer than $u$ vertices in $P_k$, we get 
$$|\mathcal{E}(H)|\leq c_3 \left(n_1\dots n_{k-1}n_k^{1-1/e_k}+ n_k F^\varepsilon_{e_1,\dots, e_{k-1}}(n_1,\dots, n_{k-1})\right)$$
for some constant $c_3(\vec{d}, t,u)$. This finishes the proof of proposition \ref{prop-init-bound}.
\end{proof}
\begin{cor}
Theorem \ref{main thm} holds if $n_1\dots n_{k} n_i^{-1/e_i}\leq n_iF^\varepsilon_{\pi_i(\vec{e})}(\pi_i(\vec{n}))$ for some $i\leq k$. 
\end{cor} \begin{proof}
By symmetry, Proposition \ref{prop-init-bound} holds if we replace $k$ by $i$. Hence $|\mathcal{E}|\lesssim n_1\dots n_{k} n_i^{-1/e_i}+ n_i F^\varepsilon_{\pi_i(\vec{e})}(\pi_i(\vec{n})) \leq 2 n_iF^\varepsilon_{\pi_i(\vec{e})}(\pi_i(\vec{n})) \lesssim F^\varepsilon_{\vec{e}}(\vec{n})$. The last inequality follows from the fact that $n_iF^\varepsilon_{\pi_i(\vec{e})}(\pi_i(\vec{n}))$ are terms that appear in the definition  \eqref{defi_F} of $F^\varepsilon_{\vec{e}}(\vec{n})$; in fact, they are precisely the terms where $I\not\owns i$.
\end{proof}
\begin{remark}\label{assumption}
Thus from now on we can assume $n_1\dots n_{k-1} n_k n_i^{-1/e_i}\geq  n_i F^\varepsilon_{\pi_i(\vec{d})}(\pi_i(\vec{n}))$ for any $i\in[k]$. In particular, we can assume $n_j\geq   n_i^{1/e_i}$ for any $i\neq j$ (since $n_1\dots n_k n_i^{-1/e_i}\geq n_1\dots n_k n_j^{-1}$). By Lemma \ref{E_d dominant}, we can assume
$E_{\vec{e}}(\vec{n})$ is the dominant term in $F^\varepsilon_{\vec{e}}(\vec{n})$. 
\end{remark}

In the second step, 
we think of the hyperedges as incidences between $n_k$ semi-algebraic sets from $P_k$ with a grid $P_1\times\dots\times P_{k-1}$ in $\R^{d_1+\dots+d_{k-1}}$. 
Recall there are $t$ polynomials $f_1,\dots, f_t$ and a Boolean function $\phi$ such that $\vec{p}=(p_1,\dots, p_k)\in P_1\times \dots\times P_k$ is a hyperedge of $H$ if and only if $\phi(f_1(\vec{p})\geq 0,\dots, f_t(\vec{p})\geq 0)=1$.
For each $p\in P_k$, define the set of its neighbors: $\gamma_p:=\{(x_1,\dots, x_{k-1})\in\R^{d_1}\times\dots\times \R^{d_{k-1}}:  (x_1,\dots, x_{k-1},p)\in \mathcal{E}(H) \}$. It is easy to see each $\gamma_p$ is a semi-algebraic set in $\R^{d_1+\dots+d_{k+1}}$ defined by $f_1,\dots, f_t$ and $\phi$. Moreover, $\mathcal{E}$ is exactly the number of incidences between these semi-algebraic sets $\{\gamma_p\}_{p\in P_k}$ with the grid $P_1\times\dots\times P_{k-1}$, which is denoted  by $I(P_1\times\dots\times P_{k-1}, P_k)$.

We now prove $I(P_1\times\dots\times P_{k-1}, P_k)\lesssim F^\varepsilon_{\vec{e}}(\vec{n})$  by induction on $\sum_{i=1}^{k-1} e_i$ and then on $\sum_{i=1}^{k-1} n_i$.  The statement is vacuous when $\sum e_i=0$. Fix some $e_1,\dots, e_k$ and assume our inequality \eqref{main ineq} holds whenever $\sum \dim(V_i) <\sum e_i$. We now use induction on $\sum_{i=1}^{k-1} n_i$. We can choose the coefficient big enough so that it holds for small $n_1,\dots, n_k$. 
For the induction step, fix some $n_1,\dots, n_k$ and assume \eqref{main ineq} holds whenever $\sum |P_i|< \sum n_i$. Let $r>0$ be a constant  to be chosen later. By theorem \ref{poly partitioning} for each $i<k$, there exists an $r^{e_i}$-partitioning polynomial $f_i\in \R[x_1,\dots, x_{d_i}]$ with respect to $V_i$ of degree at most $C_{part} r$. 
The polynomial we use to partition the grid $P_1\times \dots\times P_{k-1}$ is
$$h(x_1,\dots, x_{d_1+\dots+d_{k-1}})= f_1(x_1,\dots, x_{d_1})f_2(x_{d_1+1},\dots,x_{d_1+d_2})\dots f_{k-1}(x_{d_1+\dots+d_{k-2}+1},\dots, x_{d_1+\dots+d_{k-1}}).$$
By theorem \ref{Tao}, for each $i<k$, $f_i$ divides $\R^{d_i}$ into $s_i=O(r^{e_i})$ cells. 
Therefore $\R^{d_1+\dots+d_{k-1}}\setminus Z(h)$ consists of $s_1\dots s_{k-1}=O(r^{e_1+\dots+e_{k-1}})$ cells, each cell contains a sub-grid of $P_1\times \dots\times P_{k-1}$ of size at most $\frac{n_1}{r^{e_1}}\times\dots\times\frac{n_{k-1}}{r^{e_{k-1}}}$. We consider 3 types of incidences in $I(P_1\times\dots\times P_{k-1},P_k)$:

\begin{itemize}
\item $I_1$ consists of the incidences $(p_1,\dots, p_k)$ where $p_i\in V_i\cap Z(f_i)$ for some $i<k$.
\item $I_2$ consists of the incidences $(p_1,\dots,p_k)$ where $(p_1,\dots,p_{k-1})$ is in a cell $\Omega$ of the partitioning of $h$ and the semi-algebraic set $\gamma_{p_k}$ fully contains $\Omega$.
\item $I_3$ consists of the incidences $(p_1,\dots,p_k)$ where $(p_1,\dots,p_{k-1})$ is in a cell $\Omega$ of the partitioning of $h$ and the semi-algebraic set $\gamma_{p_k}$ intersects $\Omega$ but does not fully contain $\Omega$ (in other words, $\gamma_{p_k}$ crosses $\Omega$, or $\gamma_{p_k}$ properly intersects $\Omega$).
\end{itemize} 

Then we have $I(P_1\times\dots\times P_{k-1},P_k)=I_1+I_2+I_3$.
\\
\\
\textbf{Bounding $I_1$:} Let $I_1^i$ denote 
  the number of incidences $(p_1,\dots, p_k)$ where $p_i$ belongs to $V_i\cap Z(f_i)$. Since $I_1\leq \sum_i I^i_1$, it is enough to bound $I^1_1$ (the bound for $I^i_1$ for $i>2$ is similar).
The points of $P_1$ participating in $I_1^1$ belong to $Z(f_1)\cap V_1$. Assume $V':=Z(f_1)\cap V_1$ has dimension $e'_1$. Since $V_1$ is irreducible and $f\notin I(V_1)$, we have $e'_1<\dim V_1=e_1$. We can partition this new variety into at most $c_1$ irreducible varieties, each of dimension at most $e_1-1$ and degree at most $c_2$ where $c_1,c_2$ only depend on $D_1, C_{part}, d_1$ and $r$. Since $e'_1+e_2+\dots+e_{k-1}<e_1+\dots+e_{k-1}$, we can apply induction hypothesis for each irreducible component and add together to get 
$I^1_1\lesssim F^\varepsilon_{e'_1, e_2,\dots, e_k}(\vec{n})$. By 
applying Lemma \ref{compare d and d-1} repeatedly and by remark \ref{assumption}, we have $F^\varepsilon_{e'_1,\dots, e_k}(\vec{n})\leq F_{e_1,\dots, e_k}(\vec{n})$. Thus $I^1_1\lesssim F^\varepsilon_{\vec{e}}(\vec{n}).$
\\
\\
\textbf{Bounding $I_2$:} 
Any cell in the partitioning using $h$ has form $\Omega=\Omega^{f_1}\cap \dots \cap \Omega^{f_{k-1}}$ where  $\Omega^{f_i}$ is some cell in the decomposition by $f_i$. For each $i<k$, the contribution to $I_2$ from all cells that satisfy $|\Omega^{f_i}\cap P_i|< u$ is bounded by $u\prod_{j\neq i} n_j$. Hence we only need to bound the contribution from cells that contain a grid of size at least $u\times \dots\times u$ (there are $k-1$ $u's$). For such a cell $\Omega$, the number of  semi-algebraic sets $\gamma_{p_k}$ that  contain  $\Omega$ is bounded by $u$, because otherwise the hypergraph would contain $K_{u,\dots,u}$. Thus the contribution to $I_2$ from this last cell type is at most $u\prod_{i<k} n_i$. In conclusion, $I_2\lesssim \prod n_i (\sum \frac{1}{n_i})< F^\varepsilon_{\vec{e}}(\vec{n})$.
\\
\\
\noindent\textbf{Bounding $I_3$:} For each $i<k$, and each $j\leq s_j$, let $n_{i,j}$ denote the number of points of $P_i$ that lies in a cell $\Omega^{f_i}_j$. Then clearly for each $i<k$, $\sum_{j=1}^{s_i} n_{i,j}\leq n_i$. For $j_i\leq s_i$, let $n_{k,j_1,\dots, j_{k-1}}$ denote the number of semi-algebraic sets $\gamma_{p_k}$ that cross the cell $\Omega^{f_1}_{j_1}\cap\dots\cap\Omega^{f_{k-1}}_{j_{k-1}}$. Using a similar argument with the previous section, by theorem \ref{Tao}, each semi-algebraic set $\gamma_{p_k}$ crosses at most $O(r^{e_1+\dots+e_{k-1}-1})$ cells, and hence $\sum n_{k,j_1,\dots, j_{k-1}}\leq r^{e_1+\dots+e_{k-1}-1} n_k$. 
Hence 
\begin{align}
I_3 &=\sum_{j_1,\dots, j_{k-1}} I(n_{1,j_1},\dots,n_{k-1,j_{k-1}} ,n_{k,j_1,\dots, j_{k-1}})\nonumber\\
&\label{eq-1}\lesssim \sum_{j_1,\dots, j_{k-1}} F^\varepsilon_{\vec{e}}(n_{1,j_1},\dots, n_{k-1, j_{k-1}}, n_{k,j_1,\dots, j_{k-1}})\\
&\label{eq-2}\leq r^{e_1+\dots+e_{k-1}} F^\varepsilon_{\vec{e}}(\frac{n_1}{r^{e_1}},\dots,\frac{n_{k-1}}{r^{e_{k-1}}},\frac{n_k}{r})
\\
& \label{eq-3}\leq r^{-\varepsilon} F^\varepsilon_{\vec{e}} (\vec{n}).
\end{align}
Here \eqref{eq-1} follows by the induction assumption for smaller $\sum n_i$, \eqref{eq-2} follows by H\"older's inequality (as all the exponents are either 1 or less than 1) and the last step \eqref{eq-3} follows from Corollary \ref{inflate by r property} (note that  by remark \ref{assumption}, we only need to care about the dominant term $E_{\vec{d}}(\vec{n})\prod_i n_i^\varepsilon)$.

Adding $I=I_1+I_2+I_3$, choosing appropriate $r$ and coefficients, we get the desired bound.
This finishes the proof of our main theorem.
\qed 
\section{Applications}\label{section:application}
\subsection{Almost-unit-area triangle problem}
Many geometric problems can be viewed as counting edges in a certain semi-algebraic hypergraph. Let us start with the unit area triangle problem: namely, given $n$ points in $\R^2$, how many triangles of area 1 can they form? Given such a point set $P$, we can construct the 3-uniform 3-partite hypergraph $(P,P,P,\mathcal{E})$, where three points form a hyperedge iff they define a triangle of unit area. This is a semi-algebraic hypergraph with bounded complexity that contains no $K_{1,2,2}$ (see lemma \ref{no K222}). Hence the number of unit area triangles, which is the number of hyperedges, is $O(n^{12/5+o(1)})$ by Theorem \ref{main thm}. This is weaker than the best bound known $O(n^{9/4})$ (see \cite{Raz unit triangle}), but more robust. For example, if we are interested in counting the number of triangles with area close to 1, say between $0.9$ and $1.1$ with an additional condition of not containing some $K_{u,u,u}$, Theorem \ref{main thm} gives us a nice bound. The hypergraph formed by all triangles with areas between 0.9 and 1.1 is still semi-algebraic with bounded complexity, and hence Theorem \ref{main thm} has the following corollary:
\begin{cor}
Given a set $S$ of $n$ points in $\R^2$ and some $u\geq 1$. Assume the hypergraph formed by all triangles with areas between 0.9 and 1.1 contains no $K_{u,u,u}$ (i.e. there  do not exist disjoint sets $A,B,C\subset S$ such that $|A|,|B|,|C|\geq u$ and each $a\in A, b\in B, c\in C$ forms a triangle of area between 0.9 and 1.1). Then the number of those triangles is $O_\varepsilon(n^{12/5+\varepsilon})$ for arbitrarily small $\varepsilon>0$.
\end{cor}

\subsection{Unit-minor problem}\label{sec:unit minor} The previous application seems somewhat artificial as we had to impose the $K_{u,\dots, u}$-free condition. In our next application, that condition is automatically satisfied.

A natural generalization of the unit area triangle problem is to ask for the maximum number of unit-volume $d$-dimensional simplices formed by $n$ points in $\R^d$ for some fixed positive integer $d$ (see \cite{Unit area prob}). The best known bound when $d=3$ is $O(n^{7/2})$ in \cite{unit-vol-3d}. 
For general $d$, we can bound that number by $nf_d(n)$ where $f_d(n)$ is the number of unit-volume simplices with a fixed vertex, say the origin. Interestingly, this is equivalent to the unit-minor problem: What is the maximum number of unit $d\times d$ minors that appear in a $d\times n$ matrix $M$ without repeated columns?\footnote{If we allow $M$ to have repeated columns, the answer is $\Theta( n^d$), trivially.} Indeed, if we regard the column vectors of $M$ as points in $\R^d$, then the $d+1$ points $0,v_1,\dots, v_d$ form a unit-volume simplex if and only if $\det(v_1,\dots, v_d)=\pm 1$. 
When $M$ is \emph{totally positive,} that is, when all minors of $M$ are strictly positive, the best know upper bounds on $f_d(n)$ are given by the following theorem:
\begin{thm}[Farber, Ray and Smorodinsky \cite{minor in matrix}, 2014]\label{farber}
Let $f^{+}_d(n)$ denote the maximum number of $d\times d$ unit minors in a totally positive $d\times n$ matrix, then $$f^+_d(n)=\begin{cases} \Theta(n^{4/3}) &\mbox{if } d=2  \\
O(n^{11/5}) &\mbox{if } d=3\\
O(n^{d-\frac{d}{d+1}}) & \mbox{if } d\geq 4.\end{cases} $$
\end{thm}
Their proof uses point-hyperplane incidences: fix $d-1$ points, then the set of all points that form a minor 1 with those vectors is a hyperplane. In general, those hyperplanes may not be distinct, and the point-hyperplane graph can contain large complete bipartite graphs $K_{u,u}$. Farber et al. avoid these issues by imposing the total positivity constraint. 
Our theorem \ref{main thm} provides another way around these issues (without requiring total positivity) to obtain non-trivial bounds on $f_d(n)$. 
\begin{thm}
\label{minor} Let $f_d(n)$ denote the maximum number of unit $d\times d$ minors in a $d\times n$ matrix without repeated columns, then\footnote{Note that this lower bound together with the upper bound from Theorem \ref{farber} imply that $f_d(n)$ and $f_d^+(n)$ are fundamentally different.}
$n^{d-2/3}\lesssim_d f_d(n)\lesssim_{d,\varepsilon} n^{d-\frac{d}{d^2-d+1}+\varepsilon}.$
\end{thm}
\begin{cor}\label{unit-volume}
The maximum number of unit volume $d$-simplices formed by $n$ points in $\R^d$ is $O(nf_d(n))=O_{d,\varepsilon}(n^{d+1-\frac{d}{d^2-d+1}+\varepsilon})$.
\end{cor}
\begin{proof}[Proof of Theorem \ref{minor}]
Consider the $d$-uniform $d$-partite hypergraph $H$ where each part is the set of  $n$ column vectors of $M$, and $d$-tuple $(v_1,\dots, v_d)$ forms a hyperedge iff $\det(v_1,\dots, v_d)=1$. Clearly this is a semi-algebraic  hypergraph with bounded complexity. The upper bound is a simple application of Theorem \ref{main thm} for $k=d$ and the following lemma:
\begin{lem}\label{no K222}
The hypergraph $H$ does not contain $K_{2,\dots, 2}$.
\end{lem}
\begin{proof}Assume there exist $2d$ distinct points (or vectors) $\{v_i^+\}_{i=1}^d$ and $\{v_i^-\}_{i=1}^d$ in $\R^d$ such that $\det(v_1^{\sigma_1},\dots, v_d^{\sigma_d})=1$ for any choice of $\sigma_i\in\{+,-\}$. By multilinearity of the determinant we have for any choice of $\sigma_2,\dots, \sigma_d$ that
$$\det(x-y, v_2^{\sigma_2},\dots, v_d^{\sigma_d})=\det(x, v_2^{\sigma_2},\dots, v_d^{\sigma_d})- \det(y, v_2^{\sigma_2},\dots, v_d^{\sigma_d}).$$
In particular, $\det(v_1^+-v_1^-, v_2^{\sigma_2},\dots, v_d^{\sigma_d})=1-1=0$. Take any $x_1$ on the line that passes through $v_1^+$ and $v_1^-$, then 
 $\det(x_1-v_1^-, v_2^{\sigma_2},\dots, v_d^{\sigma_d})=0$ because $x_1-v_1^-$ is parallel to $v_1^+-v_1^-$. Thus $\det(x_1, v_2^{\sigma_2},\dots, v_d^{\sigma_d})=1$. Similar statements hold for other indices. Therefore we must have $\det(x_1,\dots, x_d)=1$ for any $x_i$ on the line through $v_i^+$ and $v_i^-$.
 
 Take a generic hyperplane through 0 that intersects all the lines $v_i^+v_i^-$ for $i\in[d]$ (such a hyperplane always exists because we only require the hyperplane  does not contain $d$ fixed lines). Let $x_i$ be the intersection of this hyperplane with the line through $v_i^+$ and $v_i^-$. By the above argument $\det(x_1,\dots, x_d)=1$, but as the vectors $x_1,\dots, x_d$  are linearly dependent, this is a contradiction.
\end{proof} 
Note that  $H$ can contain $K_{1,u,\dots, u}$, for $u=n/d$ by choosing $P_1=\{(1,0,\dots, 0)\}$, $P_2=\{(x_2,1,0,\dots, 0)\}$, $P_3=\{(0, x_3,1,0,\dots, 0)\}$ and $P_d=\{(0,\dots, 0, x_{d},1)\}$ for $x_i\in[u]$. 
This suggests we cannot directly apply results for graphs or even $l$-uniform hypergraphs for $l<d$.

To obtain the lower bound of $f_d(M)$: pick the tight example $P_1\times P_2$ in \cite{minor in matrix} for $d=2$ on the $x_1x_2$ plane (which comes from the tight example of Szemer\'edi-Trotter theorem). Then choose $P_3,\dots, P_d$ as above. We have at least $\sim n^{4/3}\times n^{d-2}=n^{d-2/3}$ unit minors of form
$$
 \begin{pmatrix}
a_{11} & a_{12} & 0 & \dots & 0 & 0\\
a_{21} & a_{22} & x_3 & \dots & 0 & 0\\
0 & 0 & 1 &\dots &0& 0\\
\vdots & \vdots& \vdots & \ddots & \vdots &\vdots\\
0 & 0 & 0 &\dots & 1& x_{d}\\
0 & 0 & 0 & \dots & 0& 1
\end{pmatrix}.
$$
This finishes the proof of Theorem \ref{minor}
\end{proof}

\subsection{Intersection hypergraphs}
Recall the intersection hypergraph $H(S)$ of a set $S$ of geometric objects in $\R^d$ is the $d$-uniform hypergraph on the vertex set $S$ where $d$ vertices form a hyperedge if and only if the corresponding sets have nonempty intersection. In this subsection, we find a nontrivial upper bound for the number of hyperedges in $H(S)$ given it is $K_{u,\dots, u}$-free and $S$ is taken from some $s$-dim family of semi-algebraic sets.

We say $F$ is an \emph{$s-$dimensional family of semi-algebraic sets with description complexity $t$ in $\R^d$} if each object in $F$ is a semi-algebraic set in $\R^d$ determined by at most $t$ polynomials $f_1,\dots, f_t$, each of degree at most $t$ and the coefficients of $f_1,\dots, f_t$ belong to a $s$-dim variety with degree at most $t$ in $\R^{t{d+t\choose t}}$. Informally, each semi-algebraic set in $F$ is determined by $s$ parameters. 
For example, hyperplanes or half-spaces in $\R^d$ form a $d$-dimensional  family with description complexity 1, spheres in $\R^d$ form a $(d+1)$-dimensional family with description complexity 2 and $(d-1)$-dim simplices in $\R^d$ form a $d^2$-dimensional family with description complexity $t+1$ (since each simplex is determined by $d$ points in $\R^d$ and by $t+1$ linear inequalities). 

Let $S$ be a set of $n$ semi-algebraic sets taken from a $s$-dimensional family with degree $t$. 
It is not difficult to see that the intersection hypergraph $H(S)$ is a semi-algebraic hypergraph with bounded complexity (which depends on $t,s$ and $d$). Indeed,  $(S_1,\dots, S_d)\in S^d$ forms a hyperedge if and only if there exists some $y\in R^d$ such that $y\in S_i$ for all $i\in[d]$. This is the projection of the semi-algebraic set $T=\{(y,S_1,\dots, S_d): y\in S_i, \forall i\in [d]\}$ along the first axis, hence remains a semi-algebraic set with description complexity only depending on that of $T$ and $d$ by Theorem 2.2.1 in \cite{real alg geo}. 
 Applying Theorem  \ref{main thm} to this hypergraph we get the following bound on the number of hyperedges in $H(S)$:
\begin{cor}
Let $S$ be $n$ semi-algebraic sets taken from an $s$-dimensional family with degree $t$ in $\R^d$. If $H(S)$ is $K_{u,\dots, u}$-free for some $u$, then $H(S)$ has $O_{t,d,s,u,\varepsilon}(n^{d-\frac{d}{(d-1)s+1}+\varepsilon})$ hyperedges. In particular, $n$ spheres in $\R^d$ form at most $O_{d,u, \varepsilon}(n^{d-1/d+\varepsilon})$ intersections if their intersection hypergraph is $K_{u,\dots, u}$ -free.
\end{cor}

More generally, we can extend this result to counting intersections among different families of sets, such as intersections between triangles and line segments in $\R^3$. More precisely, for $i=1,\dots, k$, let $P_i$ be $n_i$ semi-algebraic sets taken from a $s_i$-dim family in $\R^d$. Their intersection hypergraph is defined on $P_1\times\dots\times P_k$ where $(p_1,\dots, p_k)$ form a hyperedge if and only if they have a nonempty intersection. If this hypergraph is $K_{u,\dots, u}$-free for some $u>0$ then the number of intersections is $O(F^\varepsilon_{s_1,\dots, s_k}(n_1,\dots, n_k))$.

\section{Discussion}\label{section:discussion}

Perhaps the most important question raised but left unanswered by this paper is whether Theorem \ref{main thm} is tight.  When $k=2$, or $H$ is a graph, it is known to be tight when $d_1=d_2=2$ and almost tight for $d_1=d_2=d\geq 3$ (for example see \cite{Sheffer-lower bound}). No tight example is known for hypergraphs.
Another interesting question is to know the dependency of $u$ in the expression: does the theorem say anything meaningful when $u$ increases with $n$, such as $u=\Theta(\log n)$?

Furthermore, we feel that the applications of our main theorem are still largely unexplored. While it immediately gives bounds for many geometric problems, they are usually not the best ones (eg. the unit-triangle problem). It would be interesting to find an instance where hypergraphs are more effective than graphs. Such an example should share in common with the unit-minor problem that the constructed hypergraph contains $K_{1,u,\dots, u}$ for large $u$ but does not contain $K_{u',\dots, u'}$ for some fixed $u'$.  

In the almost-unit-area problem, without the condition no $K_{u,u,u}$ we can have $\Theta(n^3)$  triangles of area in the range $[0.9,1.1]$ by choosing 3 points reasonably far apart that form a unit area and then dividing the remaining $n-3$ points equally into small neighborhoods of those 3 points. Would we get a better result by imposing an upper bound on the ratio between the maximum and the minimum distances among the points?


Finally, we would like to improve the bound for the minor problem since the gap between the lower and upper bounds in Theorem \ref{minor} is quite large for big $d$. To improve the lower bound, instead of building from the grid example in two dimensions, we can choose points from a grid in $\R^d$, or a multiple of grids. 



\appendix
\section{Proof of generalized Erd\H{o}s's result }\label{appendix: Erdos}
 We suspect the second statement of Theorem \ref{Erdos} may be stated somewhere but since we cannot find a reference, its proof is included here for completeness. 
We use induction by $k$. It clearly holds for $k=1$. For $k\geq 2$, assume it holds for $k-1$. Let
$$Q:=\#\{(y, x_1,\dots, x_{u_1}): y\in P_2\times\dots\times P_{k}, x_i\in P_1, (x_i,y)\in \mathcal{E} \quad \forall i\in [u] \}.$$
We count $Q$ by two ways. For each choice of $(x_1,\dots, x_u)$ 
define a new hypergraph on  $P_2\times\dots\times P_k$ where $y$ is a hyperedge iff $(x_i, y)\in \mathcal{E}$ for all $i\in [u]$.  This $(k-1)$-uniform hypergraph does not contain $K_{u_2,\dots, u_k}$, hence by induction assumption:
\begin{equation}\label{eq-Erdos1}
E\lesssim {n_1\choose u_1} n_2\dots n_k\left(n_k^{-1/u_2\dots u_{k-1}}+\sum n_i^{-1}\right).
\end{equation}
On the other hand, for each $y\in P_2\times\dots\times P_k$, let $N_y$ denote the number of $x\in P_1$ such that $(x,y)\in \mathcal{E}$. Then by H\"older inequality:
\begin{equation}\label{eq-Erdos2}
E=\sum_{y} {N_y\choose u_1}\gtrsim \frac{(\sum_{y:N_y\geq u_1} N_y)^{u_1}}{(n_2\dots n_k)^{s_1-1}} \gtrsim \frac{(|\mathcal{E}|-n_2\dots n_k)^{u_1}}{(n_2\dots n_k)^{s_1-1}}.
\end{equation}
Combining \eqref{eq-Erdos1} and \eqref{eq-Erdos2} we get the desired bound for $|\mathcal{E}|$.
\qed 

\section{Proofs of lemmas in \ref{sec:property_EF} }\label{appendix: property E}
\noindent\textit{Proof of Lemma \ref{proof-matrix-equ}:}
A direct calculation yields 
\begin{align*} 
\sum_{j\neq i} d_j(1-\alpha_j)& =\sum_{j\neq i} \frac{\frac{d_j}{d_j-1}}{k-1+\sum_l \frac{1}{d_l-1}}=\frac{\sum_{j\neq i}\left[ 1+\frac{1}{d_j-1}\right]}{k-1+\sum_l \frac{1}{d_l-1}}\\
&=\frac{k-1+\sum_{j\neq i}\frac{1}{d_j-1}}{k-1+\sum_l 1/(d_l-1)}=1-\frac{\frac{1}{d_i-1}}{k-1+\sum_l 1/(d_l-1)}=\alpha_i.
\end{align*}
Rearranging we get $\sum_{j\neq i} d_j \alpha_j+ \alpha_i=\sum_{j\neq i} d_j$, true for all $i=1,\dots, k$, hence \eqref{matrix} holds.
\qed
\\
\\
\textit{Proof of Lemma \ref{compare d and d-1}:} 
By examining the formula of $F^\varepsilon_{\vec{d}}(\vec{n})$ in \eqref{defi_F}, we realize it is enough to prove $E_{\vec{d}_I-\mathfrak{e}_i}(\vec{n})\leq E_{\vec{d}_I}(\vec{n})$ whenever $i\in I$. Without loss of generality we can assume $I=[k]$ and $i=1$. In other words, we only need to prove $E_{d_1-1,d_2,\dots, d_k}(\vec{n})\leq E_{d_1,d_2,\dots, d_k}(\vec{n})$ given $n_i\geq n_j^{1/(d_j)}$ for any $j\neq i$. Let $M_1=k-1+\frac{1}{d_1-2}+\sum_{i=2}^k \frac{1}{d_i-1}$ and $M_2=k-1+\frac{1}{d_1-2}+\sum_{i=2}^k \frac{1}{d_i-1}$, we can write the inequality $E_{d_1-1,d_2,\dots, d_k}(\vec{n})\leq E_{\vec{d}}(\vec{n})$ as

\begin{align}
n_1^{1-\frac{1/(d_1-2)}{M_1}}\prod_{i=2}^k n_i^{1-\frac{1/(d_i-1)}{M_1}}& \leq n_1^{1-\frac{1/(d_1-1)}{M_2}}\prod_{i=2}^k n_i^{1-\frac{1/(d_i-1)}{M_2}}\nonumber\\
\iff \prod_{i=2}^k n_i^{\frac{1}{d_i-1}(\frac{1}{M_2}-\frac{1}{M_1})}&\leq n_1^{\frac{1}{M_1(d_1-2)}-\frac{1}{M_2(d_1-1)}}\nonumber
\\
\label{M1_M2}
\iff \prod_{i=2}^k n_i^{\frac{1}{(d_i-1)M_1M_2(d_1-1)(d_1-2)}}&\leq n_1^{\frac{k-1+\sum{i>1}\frac{1}{d_i-1}}{M_1M_2(d_1-1)(d_1-2)}}\\
\label{M1_M2_2}
\iff\prod_{i\geq 2} n_i^{1/(d_i-1)}&\leq n_1^{\sum_{i\geq 2} d_i/(d_i-1)}
\end{align}
In \eqref{M1_M2} we use $M_1-M_2=\frac{1}{d_1-2}-\frac{1}{d_1-1}=\frac{1}{(d_1-1)(d_2-1)}$ and $M_2(d_1-1)-M_1(d_1-2)=k-1+\sum_{i=2}^k \frac{1}{d_i-1}$. Take both sides to the power of $M_1M_2(d_1-1)(d_1-2)$ we get \eqref{M1_M2_2}, which holds because $n_i\leq n_1^{d_i}$ for any $i\geq 2$ by assumption.
\qed 
\\
\\
\textit{Proof of Lemma \ref{E_d dominant}:}
To prove $E_{\vec{d}}(\vec{n})\prod n_i^\varepsilon$ is the dominant term, it is enough to prove for any non-empty $I\subset [k]$:   \begin{equation}\label{compare E_d and E_d_I}
E_{\vec{d}}(\vec{n})\geq E_{\vec{d}_I}(\vec{n}_I)\prod_{i\notin I} n_i.
\end{equation}
\noindent\textit{Claim:} For each $i\in [k]$,   if $n_i^{-1/d_i}\prod_j n_j\geq  n_i E_{\pi_i(\vec{d})}(\pi_i(\vec{n})$ then $E_{\vec{d}}(\vec{n})\geq n_i E_{\pi_i(\vec{d})}(\pi_i(\vec{n})$ .

\noindent\textit{Proof of claim:} Both inequalities are equivalent to 
$\prod_{j\neq i} n_j^\frac{1/(d_j-1)}{k-2+\sum_{l\neq i}1/(d_l-1)}\geq n_i^{1/d_i}$ by rearrangement.\qed

We now prove  \eqref{compare E_d and E_d_I} via induction by $|I|$. By the claim, it holds whenever $|I|=k-1$ because our assumption \eqref{assume VC bound dominate} implies $n_i^{-1/d_i}\prod_i n_i\geq  n_i E_{\pi_i(\vec{d})}(\pi_i(\vec{n})$. Assume \eqref{compare E_d and E_d_I} holds for any $I\subset [k]$ with $|I|=l$, we will prove it holds for any $I\subset[k]$ such that $|I|=l-1$. Take some $i\notin I$ and let $J=I\cup\{i\}$. By induction assumption $E_{\vec{d}}(\vec{n})\geq E_{\vec{d}_{J}} (\vec{n}_J)\prod_{j\notin J} n_j$. By our assumption
\eqref{assume VC bound dominate}, $n_i^{-1/d_i}\prod_j n_j\geq E_{\vec{d}_I}(\vec{n}_I)\prod_{j\notin I} n_j$; dividing both sides by $\prod_{j\neq J} n_j$ we get $n_i^{-1/d_i}\prod_{j\in J} n_j\geq E_{\vec{d}_I}(\vec{n}_I) n_i$. Applying
the above claim for $J$ instead of $[k]$, we get $E_{\vec{d}_{J}} (\vec{n}_J)\geq n_i E_{\vec{d}_I}(\vec{n}_I)$ and thus $E_{\vec{d}}(\vec{n})\geq E_{\vec{d}_{J}} (\vec{n}_J)\prod_{j\notin J} n_j\geq E_{\vec{d}_{I}} (\vec{n}_I)\prod_{j\notin I} n_j$. This means \eqref{compare E_d and E_d_I} also holds for $I$.
\qed

\end{document}